\documentclass{amsart}

\usepackage{amsmath}
\usepackage{amssymb}
\usepackage{amsthm}
\usepackage{enumerate}
\usepackage{graphics}
\usepackage{pifont}
\usepackage{epsfig} 
\usepackage{setspace} 

\newtheorem{LinkIff}{Proposition}[section]
\newtheorem{DisToCts}[LinkIff]{Proposition}
\newtheorem{PositiveTranslation}[LinkIff]{Theorem}
\newtheorem{Definition:Cogenerator}[LinkIff]{Definition}
\newtheorem{HaakLemma}[LinkIff]{Lemma}
\newtheorem{StrongOperatorConvergence}[LinkIff]{Lemma}
\newtheorem{StrongConvergenceContraction}[LinkIff]{Lemma}
\newtheorem{ResolventLink}[LinkIff]{Proposition}
\newtheorem{PositiveTranslation2}[LinkIff]{Theorem}
\newtheorem{RightShiftCounter}[LinkIff]{Theorem}

\begin{document}
\title{$\alpha$-admissibility of the right-shift semigroup on $L^2(\mathbb{R}_+)$}

\author{Andrew Wynn}
\email{andrew.wynn@sjc.ox.ac.uk}
\address{St John's College, Oxford, OX1 3JP}
\subjclass[2000]{32A35, 32A36, 47D06}
\keywords{Admissibility, $C_0$-semigroups, Bergman spaces}

\begin{abstract}
It is shown that the right shift semigroup on $L^2(\mathbb{R}_+)$ does not satisfy the weighted Weiss conjecture for $\alpha \in (0,1)$. In other words, $\alpha$-admissibility of scalar valued observation operators cannot always be characterised by a simple resolvent growth condition. This result is in contrast to the unweighted case, where $0$-admissibility can be characterised by a simple growth bound. The result is proved by providing a link between discrete and continuous $\alpha$-admissibility and then translating a counterexample for the unilateral shift on $H^2(\mathbb{D})$ to continuous time systems. 
\end{abstract}

\maketitle
\section{Introduction}
In this paper infinite dimensional linear systems of the form 
\begin{equation} \label{continuoussystem}
\begin{array}{cclllc}
\dot x(t) &= &Ax(t), && x(0)=x_0, &t\geq 0,\\
y(t)&=&Cx(t),  &&t \geq 0,
\end{array}
\end{equation}
are studied. The operator $A$ is the generator of a $C_0$-semigroup on a Hilbert space $X$ and $C \in \mathcal{L}(D(A),Y)$ is an observation operator taking values in another Hilbert space $Y$. The reason for considering this abstract framework is that many different linear systems can be written in the form (\ref{continuoussystem}) and hence, results proved about the abstract system have the potential to be applied to a wide range of examples. Indeed, there is much literature describing how certain PDEs can be written in this abstract form, for example, see \cite{Pritchard, Salamon, Staffans, WeissSystems}. 

In order to be able to consider a larger class of linear systems, the observation operator $C$ may be assumed to be unbounded. The assumption $C \in \mathcal{L}(D(A),Y)$ means that there exists a constant $k>0$ such that 
\begin{equation*}
\|Cx\|_Y \leq k (\|x\|_X + \|Ax\|_X), \qquad x \in D(A).
\end{equation*}
Unfortunately, allowing $C$ to be unbounded means that the output map $(y(t))_{t\geq0}$ may not be well defined, since the solution $x(t)=T(t)x_0$ to (\ref{continuoussystem}) may not always be contained in $D(A)$. To avoid this situation it is often assumed that $C$ is \textit{admissible} for $A$ in the sense that there exists a constant $M>0$ such that
\begin{equation}
\|CT(\cdot)x_0\|_{L^2(\mathbb{R}_+,Y)} \leq M \|x_0\|_X, \qquad x_0 \in D(A).
\end{equation}
If $C$ is admissible for $A$ then the map $\Psi:D(A) \rightarrow L^2(\mathbb{R}_+,Y)$ given by $(\Psi x_0)(t):=CT(t)x_0, t\geq 0$ extends to a bounded linear operator on the whole space $X$ and if this is the case the output map can be interpreted as $y=\Psi x_0$. Admissibility of observation operators is a well studied concept (see e.g. \cite{Pritchard, Staffans, WeissIsrael}) and an excellent overview of the subject can be found in \cite{PartingtonReview}.

In \cite{HaakLeMerdy} a generalisation of admissibility is introduced. An observation operator $C$ is said to be $\alpha$-admissible for $A$ if there exists a constant $M>0$ such that 
\begin{equation} \label{CtsAdmis}
\int_0^\infty t^\alpha \|CT(t)x_0\|_Y^2 dt \leq M^2 \|x_0\|_X^2, \qquad x_0 \in D(A). 
\end{equation}
It is shown in \cite{HaakLeMerdy}, see also \cite{HaakKunst1}, that if $C$ is $\alpha$-admissible for $A$ then
\begin{equation} \label{CtsRes}
\sup_{\lambda \in \mathbb{C}_+} (\text{Re} \lambda)^{\frac{1-\alpha}{2}} \|CR(\lambda,A)\|_{\mathcal{L}(X,Y)} < \infty.
\end{equation}
This observation was first made in the case $\alpha=0$ by George Weiss in \cite{WeissTwoConjectures} where he conjectured that the converse statement ((\ref{CtsRes}) $\Rightarrow$ (\ref{CtsAdmis}) for $\alpha=0$) was also true. This problem received much attention and was shown to be false in general \cite{Harper2, RightShiftCounter, WeakZwart, JacobCounterexample}, but true in a number of interesting cases. For example, in the case $\alpha=0$, if $C$ is a scalar valued observation operator then it has been shown that (\ref{CtsRes}) $\Rightarrow$ (\ref{CtsAdmis}) if: $A$ is a normal operator generating a bounded $C_0$-semigroup \cite{WeissPowerful}; $A$ generates the right-shift semigroup  \cite{RightShift} on $L^2(\mathbb{R}_+)$; and if $A$ generates a contractive $C_0$-semigroup \cite{Contractions}. 

For $\alpha \in (-1,1)$, the question of whether (\ref{CtsRes}) implies (\ref{CtsAdmis}) is known as the {\em weighted Weiss conjecture}. It was shown in \cite{HaakLeMerdy} that if $A$ generates a bounded analytic $C_0$-semigroup, then $\alpha$-admissibility of any observation operator $C \in \mathcal{L}(D(A),Y)$ is equivalent to (\ref{CtsRes}) if and only if $A^{\frac12}:D(A) \rightarrow X$ is $0$-admissible for $A$. However, in the non-analytic case the behaviour of the problem can depend on the parameter $\alpha$. If $A$ is a normal operator generating a bounded $C_0$-semigroup then: if $\alpha \in (0,1)$, properties (\ref{CtsAdmis}) and (\ref{CtsRes}) are equivalent for any scalar valued observation operator \cite{Wynn1}; but for $\alpha \in (-1,0)$, there exist normal semigroup generators and scalar valued observation operators \cite{Wynn2} for which (\ref{CtsRes}) $\not \Rightarrow$ (\ref{CtsAdmis}).

The main result of this paper, Theorem \ref{6:RightShiftCounter}, concerns the right shift semigroup on $L^2(\mathbb{R}_+)$ given by
\begin{equation} \label{RShiftIntro}
(S(t)f)(\tau):= \left\{ \begin{array}{lcccc} f(\tau-t), & \tau \geq t;\\ 0, & \tau < t; \end{array} \right., \qquad f \in L^2(\mathbb{R}_+).
\end{equation}
There are two important reasons for studying admissibility of (\ref{RShiftIntro}). First, the behaviour of the weighted Weiss conjecture is known for normal operators \cite{Wynn1,Wynn2} and (\ref{RShiftIntro}) is one of the simplest {\em non-normal} operators. Second, the fact that (\ref{RShiftIntro}) satisfies the unweighted Weiss conjecure (i.e. (\ref{CtsRes}) $\Rightarrow$ (\ref{CtsAdmis}) if $\alpha=0$) is the vital step in proving that the unweighed Weiss conjecture is true for {\em all} contractive $C_0$-semigroups \cite{Contractions}. 

For $\alpha \in (0,1)$, the behaviour of the weighted Weiss conjecture for (\ref{RShiftIntro}) is in contrast to the unweighted case $\alpha =0$: it is shown in Theorem \ref{6:RightShiftCounter} that there exist scalar valued observation operators for which (\ref{CtsRes}) $\not \Rightarrow$ (\ref{CtsAdmis}). To help prove Theorem \ref{6:RightShiftCounter} it will be shown in \S 2 that for $\alpha \in (0,1)$, $\alpha$-admissibility can be linked to the concept of discrete $\alpha$-admissibility studied in \cite{HarperWeiss, Wynn1, Wynn2}. 

If $X$ and $Y$ are Hilbert spaces and $T \in \mathcal{L}(X)$ is a contraction $(\|T\|_{\mathcal{L}(X)} \leq 1)$, an observation operator $D \in \mathcal{L}(X,Y)$ is said to be \textit{discrete $\alpha$-admissible} for $T$ if there exists a constant $M>0$ such that 
\begin{equation} \label{DisAdmis}
\sum_{n=0}^\infty (1+n)^\alpha \|DT^n x\|_Y^2 \leq M^2 \|x\|_X^2, \qquad x \in X.
\end{equation}
It is shown in \cite{Wynn1} that if $\alpha \in (-1,1)$ and (\ref{DisAdmis}) holds then
\begin{equation} \label{DisRes}
\sup_{\omega \in \mathbb{D}} (1-|\omega|^2)^{\frac{1-\alpha}{2}} \|D(I-\bar \omega T)^{-1}\|_{\mathcal{L}(X,Y)} < \infty,
\end{equation}
and it is again natural to attempt to determine the class of operators $(D,T)$ for which the reverse implication (\ref{DisRes}) $\Rightarrow$ (\ref{DisAdmis}) holds. The question of whether this implication holds for certain types of operator is known as the {\em discrete weighted Weiss conjecture}. In the case $\alpha=0$, it is shown in  \cite{HarperWeiss} that (\ref{DisRes}) $\Rightarrow$ (\ref{DisAdmis}) for any scalar valued observation operator $D \in \mathcal{L}(X,\mathbb{C})$. Furthermore, this was shown to be directly equivalent to the fact that (\ref{CtsAdmis}) $\Leftrightarrow$ (\ref{CtsRes}) for generators of contractive $C_0$-semigroups and scalar valued observation operators $C \in \mathcal{L}(D(A),\mathbb{C})$. 

In the case $\alpha \in (0,1)$, it is shown in \cite{Wynn1} that the question of whether (\ref{DisRes}) $\Leftrightarrow$ (\ref{DisAdmis}) for normal operators $T$ is equivalent to the problem of whether 
(\ref{CtsRes})$\Leftrightarrow$ (\ref{CtsAdmis}) for normal $C_0$-semigroup generators $A$. However, since the right-shift semigroup is not a normal operator, this result cannot be applied here. For this reason \S \ref{TRANS} of this paper will show how to relate continuous $\alpha$-admissibility to discrete $\alpha$-admissibility if $T$ is any contraction operator and $A$ generates a contractive $C_0$-semigroup. This result will then be applied in \S \ref{RSS} to translate a discrete time counterexample from \cite{Wynn2} to the right-shift semigroup.

\section{Translating $\alpha$-admissibility} \label{TRANS}
It is shown in \cite{Wynn1} that discrete $\alpha$-admissibility is closely related to the weighted Bergman space $\mathcal{A}_{\alpha-1}^2(\mathbb{D})$, which contains analytic functions $f:\mathbb{D} \rightarrow \mathbb{C}$ for which 
\begin{equation*}
\|f\|_{\mathcal{A}_{\alpha-1}^2(\mathbb{D})}^2 := \int_{\mathbb{D}} |f(z)|^2 (1-|z|^2)^{\alpha-1} dA(z) < \infty,
\end{equation*}
where $dA(z)$ is Lebesgue area measure on $\mathbb{D}$.
 The reproducing kernels $k_\omega^{\alpha-1} \in \mathcal{A}_{\alpha-1}^2(\mathbb{D})$ are given by
\begin{equation} \label{RepSeries}
k_\omega^{\alpha-1}(z)  :=  \frac{1}{(1-\bar \omega z)^{1+\alpha}} = \sum_{n=0}^\infty \alpha^{(n)} \bar \omega^n z^n, \qquad \omega \in \mathbb{D}, z \in \mathbb{D},
\end{equation}
where $\alpha^{(0)}:=1$ and $\alpha^{(n)}:= \frac{1}{n!} \prod_{i=1}^n (i+\alpha), n \geq 1$. Moreover, 
\begin{equation} \label{4:IntIterate}
\alpha \int_0^1 (1-x)^{\alpha-1} x^n dx = \frac{1}{\alpha^{(n)}}, \qquad n \in \mathbb{N}, \alpha >0.
\end{equation}
Continuous $\alpha$-admissibility is related to the weighted Bergman space $\mathcal{A}_{\alpha-1}^2(\mathbb{C}_+)$, which contains those analytic functions $F:\mathbb{C}_+ \rightarrow \mathbb{C}$ for which
\begin{equation*}
\|F\|_{\mathcal{A}_{\alpha-1}^2(\mathbb{C}_+)}^2 := \int_0^\infty \!\!\! \int_{\! -\infty}^\infty x^{\alpha-1} |F(x+iy)|^2 dx dy < \infty.
\end{equation*}
The reproducing kernels for $\mathcal{A}_{\alpha-1}^2(\mathbb{C}_+)$ are given by $K_\lambda^{\alpha-1}(z):=(\bar \lambda +z)^{-(1+\alpha)},  z, \lambda \in \mathbb{C}_+$
and are related to the reproducing kernels for $\mathcal{A}_{\alpha-1}^2(\mathbb{D})$ by the identity
\begin{equation} \label{RepRel}
k_\omega^{\alpha-1}(\mathcal{M}z) = \left( \frac{1+z}{1+\bar \omega} \right)^{1+\alpha} K_{\mathcal{M}\omega}^{\alpha-1} (z), \qquad \omega \in \mathbb{D}, z \in \mathbb{C}_+,
\end{equation}
where $\mathcal{M}z:=(1-z)(1+z)^{-1}$. 
The following identity will also be useful.
\begin{equation} \label{Ident}
(1-|\omega|^2) = \text{Re}\left( \mathcal{M}\omega \right) |1+\omega|^2, \qquad \omega \in \mathbb{D}.
\end{equation}
The following proposition appears in \cite{Wynn1} for scalar valued observation operators, but it is not difficult to extend it to apply to observation operators of the form $C \in \mathcal{L}(D(A),Y)$. 

\begin{LinkIff} \label{LinkIff}
Let $\alpha \in (0,1)$ and let $X$ and $Y$ be Hilbert spaces. Suppose that $A$ generates a contractive $C_0$-semigroup on $X$ and $C \in \mathcal{L}(D(A),Y)$. Define $T:=(I+A)(I-A)^{-1} \in \mathcal{L}(X)$ and $D:=  C(I-A)^{-(1+\alpha)} \in \mathcal{L}(X,Y)$. Then $C$ is $\alpha$-admissible for $A$ if and only if $D$ is discrete $\alpha$-admissible for $T$. 
\end{LinkIff}

\begin{DisToCts} \label{DisToCts}
Let $\alpha \in (0,1)$ and let $X$ and $Y$ be Hilbert spaces. Suppose that $A$ generates a contractive $C_0$-semigroup on $X$ and $C \in \mathcal{L}(D(A),Y)$. Then if $(\ref{CtsRes})$ holds it follows that  $D:=C(I-A)^{-(1+\alpha)} \in \mathcal{L}(X,Y)$ and $T:=(I+A)(I-A)^{-1} \in \mathcal{L}(X)$ satisfy $(\ref{DisRes})$.
\end{DisToCts}
\begin{proof}
If (\ref{CtsRes}) holds there exists a constant $k>0$ such that for any $\lambda \in \mathbb{C}_+$,
\begin{align}
\|CR(\lambda,A)^{1+\alpha}\|_{\mathcal{L}(X,Y)} &\leq \|CR(\lambda,A)\|_{\mathcal{L}(X,Y)} \|R(\lambda,A)^{\alpha}\|_{\mathcal{L}(X)} \nonumber \\
					 &\leq  \frac{k}{(\text{Re}\lambda)^{\frac{1+\alpha}{2}}}.\label{1PlusAlphaResolvent} 
\end{align}
It follows from (\ref{RepRel}), (\ref{Ident}) and (\ref{1PlusAlphaResolvent}) that for any $\omega \in \mathbb{D}$,
\begin{align}
\|D(I - \bar \omega T)^{-(1+\alpha)}\|_{\mathcal{L}(X,Y)} & = \frac{ \|CR(\mathcal{M}\bar \omega , A)^{1+\alpha}\|_{\mathcal{L}(X,Y)}}{|1+\bar \omega|^{1+\alpha}} \nonumber \\
& \leq  \frac{k}{(1-|\omega|^2)^\frac{1+\alpha}{2}}. \label{bounds}
\end{align}	
Therefore,
\begin{align*}
\int_0^1 (1-x)^{\alpha-1} (I - x \bar \omega T)^{-(1+\alpha)} dx 
& = \int_0^1 (1-x)^{\alpha-1} \sum_{n=0}^\infty \alpha^{(n)}  x^n \bar \omega^n T^n  dx\\
\text{(Fubini)} & = \sum_{n=0}^\infty \bar \omega^n T^n \alpha^{(n)} \int_0^1 (1-x)^{\alpha-1} x^n dx \\
\text{(by (\ref{4:IntIterate}))} & =  \frac{1}{\alpha} (I-\bar \omega T)^{-1}.
\end{align*}
Hence, for any $\omega \in \mathbb{D}$,
\begin{align*}
 \|D (I- \bar \omega T)^{-1}\|_{\mathcal{L}(X,Y)}
 & =    \alpha \left\| \int_0^1 (1-x)^{\alpha-1} D(I- \bar \omega x T)^{-(1+\alpha)} dx \right\|_{\mathcal{L}(X,Y)} \\													 
 &\leq  \alpha \int_0^1 (1-x)^{\alpha-1} \|D(I-\bar \omega x T)^{-(1+\alpha)}\|_{\mathcal{L}(X,Y)} dx\\						 \text{(by (\ref{bounds}))}
  & \leq  \alpha k \int_0^1 \frac{(1-x)^{\alpha-1}}{(1-|\bar \omega x|^2)^{\frac{1+\alpha}{2}}} dx\\
		& \leq  \alpha k \int_0^1 \frac{(1-x)^{\alpha-1}}{(1-|\bar \omega x|)^{\frac{1+\alpha}{2}}} dx.
\end{align*}
Now, setting $x=1-(1-|\omega|)s$ implies that
\begin{align*}
\|D (I-  \bar \omega T)^{-1}\|_{\mathcal{L}(X,Y)}
& \leq     \alpha k  \int_0^{\frac{1}{1-|\omega|}} \frac{(1-|\omega|)^\alpha s^{\alpha-1}}{((1-|\omega|)(1+|\omega|s))^{\frac{1+\alpha}{2}}} ds \\
																																			& =     \frac{\alpha k}{(1-|\omega|)^{\frac{1-\alpha}{2}}} \int_0^{\frac{1}{1-|\omega|}} \frac{ s^{\alpha-1}}{(1+|\omega|s)^{\frac{1+\alpha}{2}}} ds\\
																																			& \leq  \frac{k_\alpha}{(1-|\omega|^2)^{\frac{1-\alpha}{2}}}.
\end{align*}
\end{proof}

It is now possible to deduce information about the continuous weighted Weiss conjecture from information about the discrete weighted Weiss conjecture. 

\begin{PositiveTranslation} \label{PositiveTranslation}
Let $\alpha \in (0,1)$. Suppose that $A$ is the generator of a contractive $C_0$-semigroup on a Hilbert space $X$ and let $T:=(I+A)(I-A)^{-1} \in \mathcal{L}(X)$. Then if $(\ref{DisRes}) \Rightarrow (\ref{DisAdmis})$ for every observation operator $D \in \mathcal{L}(X,Y)$, it follows that $(\ref{CtsRes}) \Rightarrow (\ref{CtsAdmis})$ for every observation operator $C \in \mathcal{L}(D(A),Y)$. 
\end{PositiveTranslation}
\begin{proof}
Suppose that $Y$ is a Hilbert space and $C \in \mathcal{L}(D(A),Y)$. If $(A,C)$ satisfy (\ref{CtsRes}), Proposition \ref{DisToCts} implies that $(D,T)$ satisfy (\ref{DisRes}), where $D:=C(I-A)^{-(1+\alpha)}$. By assumption, it follows that $D$ is discrete $\alpha$-admissible for $T$ and Lemma \ref{LinkIff} implies that $C$ is continuous $\alpha$-admissible for $A$. 
\end{proof}

Given a pair of operators $(T,D)$ related to a discrete time system, the obvious way to link them to a pair of continuous time operators is to assume that $T$ is a cogenerator.

\begin{Definition:Cogenerator}
A contraction operator $T$ on a Banach space $X$ is said to be a cogenerator of a contractive $C_0$-semigroup if there exists a semigroup generator $A$ on $X$ for which 
\begin{equation} \label{4:Cogenerator}
T = (I+A)(I-A)^{-1}.
\end{equation}
\end{Definition:Cogenerator}
Not all contractive operators $T$ are cogenerators of semigroups. Indeed, it is shown in \cite{SzNagy} that a contraction operator $T \in \mathcal{L}(X)$ is a cogenerator of a contractive $C_0$-semigroup if and only if $-1$ is not an eigenvalue of $T$. 

The following results concerning fractional powers of operators will be needed in order to link the continuous and discrete weighted Weiss conjectures. The author would like to thank Bernhard Haak for pointing out the following lemma. 

\begin{HaakLemma}  [{\cite{HaakThesis}, Lemma 1.3.6}] \label{4:HaakLemma}
Suppose that $A$ is the generator of a bounded $C_0$-semigroup on a Banach space $X$. Then for any $\lambda \in \mathbb{C}_+$ and $0 < \alpha < \beta$,
\begin{equation*}
R(\lambda,A)^{\alpha} = B(\alpha,\beta-\alpha)^{-1} \int_0^\infty \mu^{\beta-\alpha-1} R(\lambda+\mu,A)^{\beta} d\mu,
\end{equation*}
where $B(x,y):=\frac{\Gamma(x)\Gamma(y)}{\Gamma(x+y)}$ is the standard beta function.
\end{HaakLemma}

\begin{StrongOperatorConvergence} \label{4:StrongOperatorConvergence}
Let $\alpha \in (0,1)$ and suppose that $A$ is the generator of a contractive $C_0$-semigroup on a Banach space $X$. Then for any constant $a \geq 0$ and any $x \in X$,
\begin{equation*}
\|\lambda^{1+\alpha} R(a+\lambda,A)^{1+\alpha}x - x\|_X \rightarrow 0, \qquad \lambda \rightarrow \infty.
\end{equation*}
\end{StrongOperatorConvergence}
\begin{proof}
For $x \in D(A)$, it is shown in (\cite{Engel}, p.73) that
\begin{equation} \label{LResolvent}
\|\lambda R(\lambda,A)x-x\|_X \leq \frac{ \|Ax\|_X}{\lambda}, \qquad \lambda >0.
\end{equation}
Fix a constant $a \geq 0$. Then $R(\lambda+a,A)x = (\lambda+a) R(\lambda+a,A)^2x - R(\lambda+a,A)^2 Ax$ and it follows from (\ref{LResolvent}) that
\begin{align}
\|\lambda^2 R(a+\lambda,A)^2 x - x \|_X & \leq   \|a \lambda R(a+\lambda,A)^2 x\|_X \nonumber \\
& \qquad \qquad + \|\lambda R(a+\lambda,A)x - x \|_X \nonumber  \\
& \qquad \qquad \qquad \qquad + \|\lambda R(a+\lambda,A)^2 Ax \|_X \nonumber \\
& \leq  \frac{a \lambda \|x\|_X}{(a+\lambda)^2} + \frac{\|Ax\|_X}{(a+\lambda)} + \frac{\lambda \|Ax\|_X}{(a+\lambda)^2}\nonumber \\
& \leq  \frac{c \|x\|_{D(A)}}{\lambda}, \qquad \lambda >0. \label{4:ResolventSquare}
\end{align}
An application of Lemma \ref{4:HaakLemma} gives
\begin{align*}
\|\lambda^{1+\alpha} R(a+\lambda,A)^{1+\alpha}x -x \|_X 
 &=     \left\| c_\alpha \lambda^{1+\alpha} \int_0^\infty \mu^{-\alpha} R(a+\lambda+\mu,A)^2 x d\mu - x \right\|_X,
\end{align*}
where
\begin{equation*}
c_\alpha^{-1} = B(1+\alpha,1-\alpha) = \lambda^{1+\alpha} \int_0^\infty \frac{\mu^{-\alpha}}{(\lambda+\mu)^2} d\mu.
\end{equation*}
Hence,
\begin{align*}
\|\lambda^{1+\alpha} R(a+\lambda,&A)^{1+\alpha}x -x \|_X	\\
& =  c_\alpha \lambda^{1+\alpha} \left\| \int_0^\infty \mu^{-\alpha} R(a+\lambda+\mu,A)^2x -\frac{ \mu^{-\alpha} x}{(\lambda+\mu)^2} d\mu \right\|_X\\
																 & \leq  c_\alpha \lambda^{1+\alpha} \int_0^\infty \frac{\mu^{-\alpha} \|(\lambda+\mu)^2 R(a+\lambda+\mu,A)^2 x - x\|_X }{(\lambda+\mu)^2}  d\mu
\end{align*}
and by (\ref{4:ResolventSquare}),
\begin{align*} 
\|\lambda^{1+\alpha} R(a+\lambda,A)^{1+\alpha}x -x \|_X	 
& \leq  c_\alpha' \lambda^{1+\alpha} \|x\|_{D(A)} \int_0^\infty \frac{\mu^{-\alpha}}{(\lambda+\mu)^3} d\mu\\
																 & =     \frac{\tilde c_\alpha \|x\|_{D(A)}}{\lambda}, \qquad \lambda >0.
\end{align*}
Since the set $\{ \lambda^{1+\alpha}R(\lambda,A)^{1+\alpha} : \lambda >0\} \subset \mathcal{L}(X)$ is bounded and $D(A)$ is dense in $X$, it follows from (\cite{Engel}, Proposition A.3) that 
\begin{equation*}
\lambda^{1+\alpha} R(a+\lambda,A)^{1+\alpha}x \rightarrow x,\qquad  \lambda \rightarrow \infty,
\end{equation*}
for any $x \in X$.
\end{proof}

\begin{StrongConvergenceContraction} \label{4:StrongConvergenceContraction}
Let $\alpha \in (0,1)$. Suppose that $T \in \mathcal{L}(X)$ is a contractive operator on a Hilbert space $X$, that is the cogenerator of a contractive $C_0$-semigroup. Then for any $x \in X$,
\begin{equation*}
\|(I+\epsilon+T)^{-(1+\alpha)}(I+T)^{1+\alpha}x - x\|_X \rightarrow 0, \qquad \epsilon \rightarrow 0^+.
\end{equation*}
\end{StrongConvergenceContraction}
\begin{proof}
Since $T$ is a semigroup cogenerator, there exists a generator $A$ of a $C_0$-semigroup such that $T=(I+A)(I-A)^{-1}$. A simple calculation shows that for any $\epsilon >0$,
\begin{align*}
(I+\epsilon+T)^{-(1+\alpha)} (I+T)^{1+\alpha} & =  2^{1+\alpha} ((2+\epsilon)I - \epsilon A)^{-(1+\alpha)}\\
														 & =  \left( 2\epsilon^{-1} \right) ^{1+\alpha} R\left(1+2\epsilon^{-1},A\right)^{1+\alpha}
\end{align*}
and the result follows from Lemma \ref{4:StrongOperatorConvergence}.
\end{proof}

In order to link the continuous and discrete weighted growth bounds, it is necessary to temporarily introduce two families of observation operators. If $D \in \mathcal{L}(X,Y)$ and $T \in \mathcal{L}(X)$ is a contraction, define for $\epsilon>0$,
\begin{eqnarray*}
D_\epsilon &:=& D(I+\epsilon+T)^{-(1+\alpha)}(I+T)^{1+\alpha} \in \mathcal{L}(X,Y);\\
C_\epsilon &:=& D(I+\epsilon+T)^{-(1+\alpha)} \in \mathcal{L}(X,Y).
\end{eqnarray*}

\begin{ResolventLink} \label{4:ResolventLink}
Let $\alpha \in (0,1)$ and let $X$ and $Y$ be Hilbert spaces. Suppose that $A$ is the generator of a contractive $C_0$-semigroup on $X$, with cogenerator $T=(I+A)(I-A)^{-1}$. Then if $D \in \mathcal{L}(X,Y)$ and $(\ref{DisRes})$ holds,  $C:=D(I+T)^{-(1+\alpha)} \in \mathcal{L}(D(A),Y)$ and 
$(A,C)$ satisfy $(\ref{CtsRes})$. 
\end{ResolventLink}
\begin{proof}
By Lemma \ref{4:StrongConvergenceContraction} and the uniform boundedness theorem, there exists a constant $k>0$ such that 
\begin{equation*}
\|(I+\epsilon+T)^{-(1+\alpha)}(I+T)^{1+\alpha}\|_{\mathcal{L}(X)} \leq k, \qquad \epsilon >0.
\end{equation*}
Then since (\ref{DisRes}) holds,
\begin{equation} \label{4:1MinusAlpha}
\sup_{\epsilon >0, \omega \in \mathbb{D}} (1-|\omega|^2)^{\frac{1-\alpha}{2}} \|D_{\epsilon}(I-\bar \omega T)^{-1}\|_{\mathcal{L}(X,Y)} < \infty. 
\end{equation}
Recalling (\ref{RepSeries}), it follows that for any $\omega \in \mathbb{D}$,
\begin{eqnarray*}
\| (I-\bar \omega T)^{-\alpha} \|_{\mathcal{L}(X)} & = & \left\| I+ \sum_{n=1}^\infty \frac{\alpha (1+\alpha) \cdots (n-1+\alpha)}{n!} \bar \omega^n T^n  \right\|_{\mathcal{L}(X)} \\
\text{(since $T$ is a contraction)} & \leq & 1+ \sum_{n=1}^\infty \frac{\alpha (1+\alpha) \cdots (n-1+\alpha)}{n!} |\omega|^n\\
\text{(see \cite{Peller}, p.733)} & \leq & 1+ \frac{1}{\Gamma(\alpha)} \sum_{n=1}^\infty n^{\alpha-1} |\omega|^n \\
& \leq & k_\alpha \sum_{n=0}^\infty (1+n)^{\alpha-1} |\omega|^{n}\\
\text{(by \cite{Wynn1}, Lemma 2.2)} & \leq & \frac{k_\alpha'}{(1-|\omega|)^{\alpha}}. 
\end{eqnarray*}
By the above inequality and (\ref{4:1MinusAlpha}),
\begin{equation} \label{1PlusAlpha}
\sup_{\epsilon >0, \omega \in \mathbb{D}} (1-|\omega|^2)^{\frac{1+\alpha}{2}} \|D_{\epsilon}(I-\bar \omega T)^{-(1+\alpha)}\|_{\mathcal{L}(X,Y)} < \infty
\end{equation}
and using (\ref{RepRel}) and (\ref{Ident}) in the same way as in the proof of Proposition \ref{DisToCts}, it can be shown that 
\begin{equation} \label{EpRes}
\sup_{\epsilon >0, \lambda \in \mathbb{C}_+} (\text{Re}\lambda)^{\frac{1+\alpha}{2}} \|C_\epsilon R(\lambda,A)^{1+\alpha}\|_{\mathcal{L}(X,Y)}  < \infty. 
\end{equation}
By Lemma \ref{4:HaakLemma}, 
\begin{align}
\|C_\epsilon R(\lambda,A)\|_{\mathcal{L}(X,Y)} & =   \alpha \left\| \int_0^\infty \mu^{\alpha-1} C_\epsilon R(\mu+\lambda,A)^{1+\alpha} d\mu \right\|_{\mathcal{L}(X,Y)} \nonumber \\
														&\leq \alpha \int_0^\infty \mu^{\alpha-1} \|C_\epsilon R(\mu+\lambda,A)^{1+\alpha}\|_{\mathcal{L}(X,Y)} d\mu \nonumber \\
\text{(by (\ref{EpRes}))}		&\leq  k \int_0^\infty \frac{\mu^{\alpha-1}}{\text{Re}(\mu+\lambda)^{\frac{1+\alpha}{2}}} d\mu \nonumber \\
														&=     k \int_0^\infty \frac{\mu^{\alpha-1}}{(\mu + \text{Re}\lambda)^{\frac{1+\alpha}{2}}} d\mu \nonumber \\
														&=     \frac{k}{(\text{Re}\lambda)^{\frac{1-\alpha}{2}}} \int_0^\infty \frac{s^{\alpha-1}}{(1+s)^{\frac{1+\alpha}{2}}} ds, \nonumber \\
														&=\frac{k_\alpha}{(\text{Re}\lambda)^{\frac{1-\alpha}{2}}}, \qquad \lambda \in \mathbb{C}_+, \epsilon >0.\label{DepsilonBound}
\end{align}														
Now let $x \in D(A) \cap \text{Ran}( R(1,A)^{1+\alpha})$. Since $x \in \text{Ran}( R(1,A)^{1+\alpha})$ there exists $y \in X$ such that $x = R(1,A)^{1+\alpha}y$ and hence,
\begin{align}
\|(C-C_\epsilon)x\|_Y  & =  \| D(I+T)^{-(1+\alpha)} R(1,A)^{1+\alpha}y - C_\epsilon x \|_Y \nonumber \\
									& = 2^{-(1+\alpha)} \|Dy -D(I+\epsilon+T)^{-(1+\alpha)}(I+T)^{1+\alpha}y\|_Y \nonumber \\
									&\leq   2^{-(1+\alpha)} \|D\|_{\mathcal{L}(X,Y)} \|y - (I+\epsilon+T)^{-(1+\alpha)}(I+T)^{1+\alpha}y\|_X \nonumber \\
									 & \rightarrow   0, \qquad \epsilon \rightarrow 0^+, \label{OperatorConvergence}
\end{align}
where the last line follows from Lemma \ref{4:StrongConvergenceContraction}.					
Since $x \in D(A)$, there exists $y' \in X$ such that $x = R(1,A)y'$ and hence,
\begin{align}
\|C_\epsilon x\|_Y &= \|C_\epsilon R(1,A)y'\|_Y \nonumber \\
\text{(by (\ref{DepsilonBound}))} & \leq  k \|y'\|_X \nonumber \\
& \leq  k \|x\|_{D(A)}, \qquad \epsilon >0. \label{OperatorBound}
\end{align}
It is easy to show that $D(A) \cap \text{Ran}( R(1,A)^{1+\alpha})$ is dense in the space $(D(A), \|\cdot\|_{D(A)})$ and hence, it follows from (\ref{OperatorConvergence}) and (\ref{OperatorBound}) that $C:=D(I+T)^{-(1+\alpha)} \in \mathcal{L}(D(A),Y)$. That $(A,C)$ satisfy  (\ref{CtsRes})  follows from (\ref{DepsilonBound})  and (\ref{OperatorConvergence}).
\end{proof}

\begin{PositiveTranslation2} \label{4:PositiveTranslation2}
Let $\alpha \in (0,1)$. Suppose that $T \in \mathcal{L}(X)$ is a contractive operator on a Hilbert space $X$, which is the cogenerator of a $C_0$-semigroup $(T(t))_{t \geq 0}$. Let $A$ be the generator of $(T(t))_{t \geq 0}$. Then if $(\ref{CtsRes}) \Rightarrow (\ref{CtsAdmis})$ for each observation operator $C \in \mathcal{L}(D(A),Y)$, it follows that $(\ref{DisRes}) \Rightarrow (\ref{DisAdmis})$ for each observation operator $D \in \mathcal{L}(X,Y)$. 
\end{PositiveTranslation2}
\begin{proof}
Let $Y$ be a Hilbert space and let $D \in \mathcal{L}(X,Y)$. If $(D,T)$ satisfy (\ref{DisRes}), then by Proposition \ref{4:ResolventLink} it follows that $C:=D(I+T)^{-(1+\alpha)} \in \mathcal{L}(D(A),Y)$ and  $(A,C)$ satisfies (\ref{CtsRes}). By assumption, $C$ is continuous $\alpha$-admissible for $A$. Now,
\begin{equation*}
C(I-A)^{-(1+\alpha)} = D(I+T)^{-(1+\alpha)}(I-A)^{-(1+\alpha)} = 2^{-(1+\alpha)} D
\end{equation*}
and by Proposition \ref{LinkIff}, $D$ is discrete $\alpha$-admissible for $T$. 
\end{proof}

\section{The right shift semigroup on $L^2(\mathbb{R}_+)$} \label{RSS}
The right shift semigroup on $L^2(\mathbb{R}_+)$ is given by
\begin{equation} \label{6:RightShiftSemi}
(S(t)f)(\tau):= \left\{ \begin{array}{lcccc} f(\tau-t), & \tau \geq t;\\ 0, & \tau < t; \end{array} \right., \qquad f \in L^2(\mathbb{R}_+).
\end{equation}
In \cite{RightShift} it is shown that the right shift semigroup satisfies the unweighted Weiss conjecture (i.e. (\ref{CtsRes}) $\Leftrightarrow$ (\ref{CtsAdmis}) in the case $\alpha =0$). The calculations in \cite{RightShift} are simplified by considering the $C_0$-semigroup
\begin{equation}
(T(t)x)(z) := e^{-zt} x_0(z), \qquad z \in \mathbb{C}_+, x_0 \in H^2(\mathbb{C}_+),
\end{equation}
which, via the Laplace transform, is unitarily equivalent to the right shift semigroup $(S(t))_{t\geq 0}$. It is easy to see that for any $\alpha \in (-1,1)$, the property of satisfying the weighted Weiss conjecture is preserved under unitary equivalence of semigroups. 

The aim is to use Theorem \ref{4:PositiveTranslation2} to translate the counterexample from \cite{Wynn2} for the unilateral shift on $H^2(\mathbb{D})$ to continuous time operators. To do this, it will be useful to link the right shift semigroup (\ref{6:RightShiftSemi}) with a $C_0$-semigroup on $H^2(\mathbb{D})$. Let $J:H^2(\mathbb{C}_+) \rightarrow H^2(\mathbb{D})$ be defined by
\begin{equation}
(Jf)(z):= \frac{\sqrt{2}}{1+z} f\left( \frac{1-z}{1+z} \right), \qquad z \in \mathbb{C}_+,f \in H^2(\mathbb{C}_+).
\end{equation}
Then $J$ is a surjective isomorphism and it is easy to show that if
\begin{equation} \label{6:Qt}
Q(t)f := e^{-\left(\frac{1-z}{1+z}\right) t}f(z), \qquad z \in \mathbb{D}, f \in H^2(\mathbb{D}), t \geq 0,
\end{equation}
then $T(t)= J^{-1} Q(t) J, t \geq 0$. Therefore, $(Q(t))_{t\geq 0}$ and $(T(t))_{t\geq0}$ are unitarily equivalent semigroups and hence, $(Q(t))_{t \geq 0}$ is unitarily equivalent to the right shift semigroup (\ref{6:RightShiftSemi}) on $L^2(\mathbb{R}_+)$.

\begin{RightShiftCounter} \label{6:RightShiftCounter}
Let $\alpha \in (0,1)$ and let $A$ be the generator of the right shift $C_0$-semigroup on $L^2(\mathbb{R}_+)$ given by $(\ref{6:RightShiftSemi})$. Then there exists $C \in \mathcal{L}(D(A),\mathbb{C})$ satisfying 
\begin{equation*}
\sup_{\lambda \in \mathbb{C}_+} (\mathrm{Re} \lambda )^{\frac{1-\alpha}{2}} \|CR(\lambda,A)\|_{L^2(\mathbb{R})^*} < \infty,
\end{equation*}
but for which $C$ is not continuous $\alpha$-admissible for $A$. 
\end{RightShiftCounter}
\begin{proof}
Let $(Sf)(z):=zf(z), f \in H^2(\mathbb{D})$ be the unilateral shift on $H^2(\mathbb{D})$. The operator
\begin{equation*}
(Bf)(z):=-\left(\frac{1-z}{1+z}\right) f(z), \qquad z \in \mathbb{D}, f \in H^2(\mathbb{D}),
\end{equation*}
is the generator of the $C_0$-semigroup $(Q(t))_{t\geq0}$ given by (\ref{6:Qt}) and a simple calculation shows that 
\begin{equation*}
S = (I+B)(I-B)^{-1}
\end{equation*}
Hence, $S$ is the cogenerator of $(Q(t))_{t\geq0}$. By (\cite{Wynn2}, Theorem 3.8), there exists an observation operator $D \in H^2(\mathbb{D})^*$ which satisfies (\ref{DisRes}) but for which $D$ is not discrete $\alpha$-admissible for $S$. By Theorem \ref{4:PositiveTranslation2}, there exists $\tilde C \in \mathcal{L}(D(B),\mathbb{C})$ for which $\tilde C$ is not continuous $\alpha$-admissible for $B$, but $(B,\tilde C)$ satisfy (\ref{CtsRes}). The result follows since the weighted Weiss conjecture is preserved under unitary equivalence and $(Q(t))_{t\geq0}$ is unitarily equivalent to the right shift semigroup (\ref{6:RightShiftSemi}).
\end{proof}

\end{document}